\documentclass{IEEEtran}

\usepackage{amsmath}
\usepackage{mathtools}
\usepackage{amssymb}
\usepackage{here}
\usepackage{mathrsfs}
\usepackage{multirow}
\usepackage{cite}

\everymath{\displaystyle}

\newtheorem{theorem}{Theorem}
\newtheorem{corollary}[theorem]{Corollary}

\newtheorem{example}[theorem]{Example}

\newtheorem{remark}[theorem]{Remark}

\newcommand{\mendth}{\hfill \ensuremath{\vartriangle}}

\DeclareMathOperator*{\col}{col}
\DeclareMathOperator{\He}{Sym}
\DeclareMathOperator*{\diag}{diag}

\DeclareMathOperator{\eps}{\varepsilon}

\DeclareMathOperator{\transp}{\intercal}
\newcommand{\gint}[2]{\ensuremath{#1,\ldots,#2}}

\title{Affine characterizations of minimum and mode-dependent dwell-times for uncertain linear switched systems}

\author{Corentin Briat$^\dag$ and Alexandre Seuret$^\ddag$
\thanks{$^\dag$Corentin Briat is with the Department of Biosystems Science and Engineering (D-BSSE), Swiss Federal Institute of Technology--Z\"{u}rich (ETH-Z), Mattenstrasse 26, 4058 Basel, Switzerland; email: {\tt  corentin@briat.info, briatc@bsse.ethz.ch}; url:~{\tt http://www.briat.info}}
\thanks{$^\ddag$Alexandre Seuret is with the Department of
Automatic Control, Gipsa-lab, 961 rue de la Houille Blanche, BP 46,
38402 Grenoble Cedex, France,
email: {\tt alexandre.seuret@gipsa-lab.grenoble-inp.fr}}
\thanks{The work of A. Seuret is supported by the FeedNetBack project, FP7-
ICT-2007-2: http://www.feednetback.eu/.}}

\begin{document}
\maketitle

\begin{abstract}
An alternative approach for minimum and mode-dependent dwell-time characterization for switched systems is derived. The proposed technique is related to Lyapunov looped-functionals, a new type of functionals leading to stability conditions affine in the system matrices, unlike standard results for minimum dwell-time. These conditions are expressed as infinite-dimensional LMIs which can be solved using recent polynomial optimization techniques such as sum-of-squares. The specific structure of the conditions is finally utilized in order to derive dwell-time stability results for uncertain switched systems. Several examples illustrate the efficiency of the approach.
\end{abstract}


\section{Introduction}

Switched systems \cite{Hespanha:99,Hespanha:01c,Daafouz:02,Liberzon:03,Geromel:06b,Zhao:08,Cai:08,Lin:09,Goebel:09} are an important subclass of hybrid systems for which the system dynamics are selected among a countable family of subsystems. They are very powerful modeling tools for several real world processes, like congestion modeling and control in networks \cite{Hespanha:01b,Shorten:06,Briat:10,Briat:11k}, switching control laws \cite{Liberzon:03}, electromechanical systems \cite{Langjord:08}, networked control systems \cite{Donkers:09}, electrical devices/circuits \cite{Almer:07,Beccuti:07}, etc. These systems exhibit interesting behaviors motivating their analysis: for instance, switching between asymptotically stable subsystems does not always result in an overall stable system \cite{Decarlo:00,Liberzon:99}. Conversely, switching between unstable subsystems may result in asymptotically stable trajectories \cite{Decarlo:00,Liberzon:99}. It is well-known that in the case of asymptotically stable linear subsystems, when their matrices commute or can be expressed in an upper triangular form via a common similarity transformation, stability of the overall switched system under arbitrary switching is actually equivalent to the existence of a quadratic common Lyapunov function \cite{Liberzon:99,Liberzon:03}. In all the other cases, the existence of a quadratic common Lyapunov function is only sufficient, albeit the existence of a common (not necessarily quadratic) Lyapunov function is necessary. When no common Lyapunov function exists, approaches based for instance on  polyhedral Lyapunov functions \cite{Molchanov:89,Blanchini:95} or switched Lyapunov functions \cite{Daafouz:02}, may be considered instead. The main difficulty arising from the use of switched Lyapunov functions lies in the discontinuities of the Lyapunov function level at switching instants. If switchings occur too often, stability may indeed be lost. When considering switchings among a family of asymptotically stable subsystems, an important notion is the notion of minimum dwell-time, which is the minimal time between two successive switchings \cite{Morse:96} such that any switching rule satisfying this minimal-dwell time condition makes the overall system asymptotically stable. This notion has been later relaxed in \cite{Hespanha:99} via the introduction of the average dwell-time. Several recent results have characterized dwell-times as semidefinite programming problems (LMIs) by using quadratic and homogeneous Lyapunov functions  \cite{Geromel:06b,Chesi:10}.

In these latter results, the minimum dwell-time computation can be performed through a mix of continuous-time and discrete-time stability conditions, reflecting then the hybrid structure of the system. The continuous-time condition accounts for the asymptotic stability of the subsystems while the discrete-time condition ensures that the Lyapunov function has sufficiently decreased between switching instants so that a positive jump of the Lyapunov function level can be tolerated. The results reported in \cite{Geromel:06b,Chesi:10} have led to dramatic improvements in terms of accuracy compared to initial minimal dwell-time \cite{Morse:96} and average dwell-time \cite{Hespanha:99,Zhao:12} results. This efficiency emphasizes the importance of considering discrete-time and mixed stability criteria for analyzing switched systems, and motivates their extension to uncertain systems. The main drawback in this type of criteria is the presence of exponential terms which make the extension to uncertain systems a difficult task. There is indeed, at this time, no efficient way for dealing with matrix uncertainties at the exponential. A second drawback is their limited application to switched linear systems.

The notions of stability under minimum \cite{Morse:96,Geromel:06b,Chesi:10} and mode-dependent \cite{Zhao:12} dwell-times are considered in this paper. Concerning minimum dwell-time, the ideas of \cite{Geromel:06b} are continued while results on mode-dependent dwell-times are directly inspired from results on discrete-time switched systems \cite{Mignone:00,Daafouz:02}. Conditions for stability with minimal and mode-dependent dwell-times are derived using a new technique initially developed for sampled-data systems \cite{Seuret:12, Seuret:11} and later extended to impulsive systems \cite{Briat:11l,Briat:13a}. It has indeed been proved in \cite{Seuret:12} that discrete-time stability is equivalent to a very particular type of continuous-time stability, showing thus that any discrete-time stability criterion always has a continuous-time interpretation, in terms of the use of specific functionals referred to as \emph{looped-functionals}. The main interest of the alternative continuous-time formulation lies in its affine dependence on the system matrices that can be exploited to derive results for uncertain systems quite easily. Based on this approach, conditions for minimal and mode-dependent dwell-times characterization of uncertain switched systems are obtained. The derived criteria are expressed as infinite-dimensional convex feasibility problems that are solved using polynomial techniques \cite{Parrilo:00, Chesi:09}. As a byproduct, the approach is also valid in the case of linear switched systems with uncertain time-varying parameters.

\textit{Outline:} The structure of the paper is as follows: in Section \ref{sec:prel} preliminary definitions and results are recalled. In Sections \ref{sec:aff} and \ref{sec:modedep}, affine conditions for the characterization of minimum dwell-time and mode-dependent dwell-time are stated. The results are then finally extended to the uncertain case in Section \ref{sec:unc}. 
Examples are considered in the related sections. 

\textit{Notations:} The sets of symmetric and positive definite matrices of dimension $n$ are denoted by $\mathbb{S}^n$ and $\mathbb{S}_{+}^n$ respectively. Given two symmetric real matrices $A,B$, $A\succ(\succeq) B$ means that $A-B$ is positive (semi)definite. The set of positive real numbers is $\mathbb{R}_+$. For a square real matrix $A$, the operator $\He(A)$ stands for the sum $A+A^T$.  The identity matrix and zero-matrix of size $n$ are denoted by $I_n$ and $0_n$, respectively. For some square matrix $M\in\mathbb{R}^{n\times n}$, we denote by $\mathcal{D}_n(M)$ the block-diagonal matrix $\diag(M,0_n,0_n)$.


\section{Preliminaries}\label{sec:prel}

\subsection{Definition of the system}

Linear switched systems of the form
\begin{equation}\label{eq:mainsyst}
\begin{array}{rcl}
  \dot{x}(t)&=&A_{\sigma(t)}x(t)\\
  x(t_0)&=&x_0
\end{array}
\end{equation}
are considered in this paper. Above, $x,x_0\in\mathbb{R}^n$ are the state of the system and the initial condition, respectively. In Sections \ref{sec:aff} and \ref{sec:modedep}, the matrices $A_i$ are assumed to be exactly known, while in Section \ref{sec:unc} uncertain convex sets $\mathcal{A}_i$ of matrices are considered. The switching signal $\sigma$ is defined as a piecewise constant function $\sigma:\mathbb{R}_+\to\{1,\ldots,N\}$. We also assume that the strictly increasing sequence of switching instants $\{t_1,t_2,\ldots\}$ is state-independent, that is the mode changes are viewed as external events. The sequence is also assumed to admit no accumulation point, i.e. $t_k\to\infty$ as $k\to\infty$. In what follows, we shall consider the following family of switching rules
\begin{equation}\label{eq:it}
  \mathbb{I}_{\bar{T}}:=\left\{\{t_1,t_2,\ldots\}:\ t_{k+1}-t_{k}\in[\bar{T},+\infty),\ k\in\mathbb{N}\right\}
\end{equation}
that satisfy a minimum dwell-time condition.



\subsection{Minimum dwell-time results}

An upper-bound on the minimum dwell-time can be determined using the following result:
\begin{theorem}[\cite{Geromel:06b}]\label{th:geromel}
Assume that for some $\bar{T}>0$, there exist matrices $P_i\in\mathbb{S}^n_{+}$, $i=\gint{1}{N}$ such that the LMIs
\begin{equation}\label{eq:ctlmis}
  A_i^{\transp}P_i+P_iA_i\prec0,\ i=\gint{1}{N}
\end{equation}
and
\begin{equation}\label{eq:exp}
  e^{A_i^{\transp}\bar{T}}P_je^{A_i\bar{T}}-P_i\prec0,\ i,j=\gint{1}{N},\ i\ne j
\end{equation}
hold. Then, for any sequence of switching instants in $\mathbb{I}_{\bar{T}}$, the equilibrium solution $x=0$ of system (\ref{eq:mainsyst}) is globally asymptotically stable.\mendth
\end{theorem}
This result is easy to interpret. To this aim, let us consider the Lyapunov function $V(x(t))=x(t)P_{\sigma(t)}x(t)$, where $P_i\in\mathbb{S}_+^n$, $i=1,\ldots,N$. In such a case, the continuous-time LMIs (\ref{eq:ctlmis}) ensure that all the subsystems are asymptotically stable and that the Lyapunov function decreases between switching instants. The discrete-time conditions (\ref{eq:exp}) capture the jumps of the Lyapunov function at switching instants. Positive jumps are tolerated provided that the Lyapunov function has sufficiently decreased during the latest continuous-time regime, or equivalently, provided that $\bar{T}>0$ is sufficiently large. Note that conditions (\ref{eq:exp}) characterize first the continuous-time evolution in mode $i$, and then the discrete transition $i\to j$. It is also possible to consider the other way around, that is the transition $j\to i$ first, and then the continuous-time evolution in mode $i$. In such a case, the following dual discrete-time condition should be considered
\begin{equation}\label{eq:exp2}
  e^{A_i^{\transp}\bar{T}}P_ie^{A_i\bar{T}}-P_j\prec0,\ i,j=\gint{1}{N},\ i\ne j
\end{equation}
instead of condition (\ref{eq:exp}) in Theorem \ref{th:geromel}. The two criteria are depicted in Fig. \ref{fig:LF} where we can see that despite being non-monotonic, the switched-Lyapunov function is able to capture information on the asymptotic stability of the system.

\begin{figure}[h]
  \centering
  \includegraphics[width=0.5\textwidth]{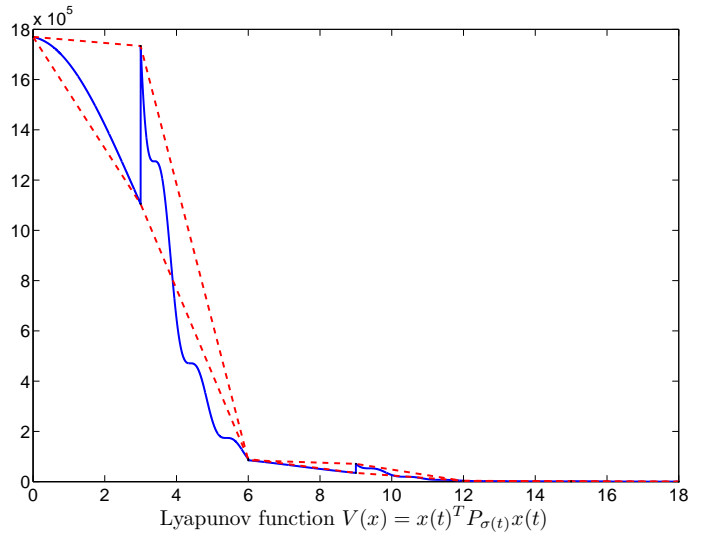}
  \caption{Evolution of a switched Lyapunov function (plain) and the two monotonically decreasing discrete-time criteria (dashed).}\label{fig:LF}
\end{figure}

It has been pointed out in \cite{Geromel:06b} that Theorem \ref{th:geromel} is more accurate than the initial results on minimum dwell-time \cite{Morse:96} and average dwell-time \cite{Hespanha:99,Zhao:12}. This fact motivates the extension of Theorem \ref{th:geromel} to the uncertain system case in Section \ref{sec:unc}.


\subsection{Mode-dependent dwell-time results}

Assume now that the system remains in mode $i$ during a period $T_i$ included in the range $T_i\in[T^{min}_i,T^{max}_i]$, $\epsilon<T^{min}_i<T^{max}_i<\infty$, $\epsilon>0$. We then have the following result:
\begin{theorem}\label{th:second}
Assume there exist matrices $P_i\in\mathbb{S}^n_{+}$, $i=\gint{1}{N}$, such that the LMIs
\begin{equation}\label{eq:djsqod}
  e^{A_i^{\transp}T_i}P_ie^{A_iT_i}-P_j\prec0,\ i,j=\gint{1}{N},\ i\ne j
\end{equation}
hold for all $T_i\in[T^{min}_i,T^{max}_i]$ . Then, the system with mode-dependent dwell-time $T_i\in[T^{min}_i,T^{max}_i]$, $i=1,\ldots,N$ is globally asymptotically stable.
\end{theorem}
\begin{proof}
  The proof is a simple adaptation of the discrete-time stability condition for switched discrete-time systems \cite{Mignone:00,Daafouz:02}. Note, however, the permutation of $P_i$ and $P_j$ compared to the usual result. This can be performed according to the discussion on Theorem \ref{th:geromel} in the previous section.
\end{proof}

In the above result, the dwell-times lie within a certain range of values. It is however possible to have the upper-bound $T_i^{max}=\infty$ for certain modes by slightly modifying the conditions. This is explained in the remark below.
\begin{remark}\label{rem:1}
When $T_i^{max}=\infty$ for some indices $i$, it is then necessary that subsystem $i$ be stable, in a similar way as in minimum dwell-time results. In this case, inequality (\ref{eq:djsqod}) considered at $T_i=\infty$ can be substituted by $A_i^{\transp}P_i+A_iP_i\prec0$. The stability conditions for subsystem $i$ then become
\begin{equation}\label{eq:rem1_1}
  e^{A_i^{\transp}T_i^{min}}P_ie^{A_iT_i^{min}}-P_j\prec0,\ j=\gint{1}{N},\ j\ne i
\end{equation}
and
\begin{equation}\label{eq:rem1_2}
A_i^{\transp}P_i+A_iP_i\prec0.
\end{equation}
A proof for this statement is similar to the one of Theorem \ref{th:geromel} and is thus omitted.
\end{remark}

It is important to mention that the provided approach for mode-dependent dwell-time is radically different from the standard approaches, such as the one in \cite{Zhao:12}. In the current approach, a discrete-time condition is used to explicitly characterize a range of values for the dwell-times, while in most of the approaches continuous-time conditions are considered. Additionally, the proposed approach does not require stability of all the subsystems, and is thus applicable to a wider class of switched systems. The price to pay is the difficulty for extending the results to uncertain systems, a problem which is resolved by the proposed approach relying on looped-functionals.

\section{Affine characterization of minimum dwell-times}\label{sec:aff}

The main drawback of Theorem \ref{th:geromel} lies in the presence of exponentials in the conditions (\ref{eq:exp})  which make the extension to uncertain systems very difficult since there is no efficient way for considering matrix uncertainties at the exponential. To overcome this difficulty, an alternative condition for Theorem \ref{th:geromel} that is affine in the system matrices is provided in this section. It is also emphasized that this novel condition can be interpreted as the non-increase condition of a certain class of looped-functional \cite{Seuret:12,Seuret:11,Briat:11l,Briat:13a} recently introduced by the authors.

\subsection{Main results}
The following result provides sufficient conditions for Theorem \ref{th:geromel}.
\begin{theorem}\label{th:equiv}
Assume there exist a scalar $\eps>0$, matrices $P_i\in\mathbb{S}^n_{+}$, $i=\gint{1}{N}$ and symmetric differentiable matrix functions $Z_{ij}:[0,\bar{T}]\to\mathbb{S}^{3n}$, $i,j=\gint{1}{N}$, $i\ne j$, verifying
\begin{equation}\label{eq:loliloolldr}
Y_2^{\transp}Z_{ij}(\bar{T})Y_2-Y_1^{\transp}Z_{ij}(0)Y_1=0
\end{equation}
where
\begin{equation}
\begin{array}{lclclcl}
  Y_1&=&\begin{bmatrix}
    I_n & 0_n\\
    I_n & 0_n\\
    0_n & I_n
  \end{bmatrix},&\quad& Y_2&=&\begin{bmatrix}
    0_n & I_n\\
    I_n & 0_n\\
    0_n & I_n
  \end{bmatrix}
\end{array}
\end{equation}
and such that the LMIs
\begin{equation}
  A_i^{\transp}P_i+P_iA_i\prec0,\ i=\gint{1}{N}
\end{equation}
\begin{equation}\label{eq:expless}
 \Psi_{ij}(\bar{T})+\He\left(Z_{ij}(\tau)\mathcal{D}_n(A_i)\right)+\dot{Z}_{ij}(\tau)\preceq0,\ i,j=\gint{1}{N},\ i\ne j
\end{equation}
hold for all $\tau\in[0,T]$ where
\begin{equation}
  \Psi_{ij}(\bar{T}):=\begin{bmatrix}
   \bar{T}(A_i^{\transp}P_i+P_iA_i) & 0_n & 0_n\\
   0_n & P_i-P_j+\eps I_n & 0_n\\
   \star & \star & 0_n
 \end{bmatrix}.
\end{equation}
Then, the switched system (\ref{eq:mainsyst}) is asymptotically stable for any sequence of switching instants in $\mathbb{I}_{\bar{T}}$ and the conditions of Theorem \ref{th:geromel} are satisfied with the same matrices $P_i$.\mendth
\end{theorem}
\begin{proof}  To see the implication, first pre- and post-multiply (\ref{eq:expless}) by $$\xi(\tau):=\col(x(\tau),x(0),x(\bar{T})),\tau\in[0,\bar{T}]$$ to obtain
\begin{equation}\label{eq:xit}
\xi(\tau)^{\transp}\Psi_{ij}(\bar{T})\xi(\tau)+\dfrac{d}{d\tau}\left[\xi(\tau)^{\transp}Z_{ij}(\tau)\xi(\tau)\right]\le0.
\end{equation}
Integrating the above inequality from $0$ to $\bar{T}$, we get
\begin{equation}\label{eq:eta}
\begin{array}{lcl}
  \eta_{ij}&:=&\int_0^{\bar{T}}\left[x(0)^{\transp}(P_i-P_j+\eps I_n)x(0)\right.\\
  &&+\left.\bar{T}\dfrac{d}{d\tau}V_i(x(\tau))\right]d\tau\\
  &&+\xi(\bar{T})^{\transp}Z_{ij}(\bar{T})\xi(\bar{T})-\xi(0)^{\transp}Z_{ij}(0)\xi(0)\le0
\end{array}
\end{equation}
where $V_i(x)=x^{\transp}P_ix$. Noting that
\begin{equation}\label{eq:toyloilol}
  \xi(0)=Y_1\begin{bmatrix}
  x(0)\\
  x(\bar{T})
\end{bmatrix},\quad \xi(\bar{T})=Y_2\begin{bmatrix}
  x(0)\\
  x(\bar{T})
\end{bmatrix}
\end{equation}
the last row of (\ref{eq:eta}) can be rewritten as
\begin{equation}
  \begin{bmatrix}
    x(\bar{T})\\
    x(0)
  \end{bmatrix}^{\transp}\left(Y_2^{\transp}Z_{ij}(\bar{T})Y_2-Y_1^{\transp}Z_{ij}(0)Y_1\right)\begin{bmatrix}
    x(\bar{T})\\
    x(0)
  \end{bmatrix}
\end{equation}
which is equal to 0 by virtue of the constraint (\ref{eq:loliloolldr}). Hence, we have
\begin{equation}
  \begin{array}{lcl}
  \eta_{ij}&=&\bar{T}\left[V_i(x(0))-V_j(x(0))+\eps||x(0)||^2_2\right.\\
  &&\left.+V_i(x(\bar{T}))-V_i(x(0))\right]\\
  &=&\bar{T}\left[V_i(x(\bar{T}))-V_j(x(0))+\eps||x(0)||^2_2\right].
  \end{array}
\end{equation}
Finally, noting that $$V_i(x(\bar{T}))-V_j(x(0))=x(0)^{\transp}\left[e^{A_i^{\transp}\bar{T}}P_ie^{A_i\bar{T}}-P_j\right]x(0),$$ then we obtain
$$x(0)^{\transp}\left[e^{A_i^{\transp}\bar{T}}P_ie^{A_i\bar{T}}-P_j\right]x(0)\le-\eps||x(0)||_2^2$$ for all $x(0)\in\mathbb{R}^n$ since (\ref{eq:xit}) is nonpositive. This then implies that (\ref{eq:exp2}) holds and shows that the feasibility of the conditions of Theorem \ref{th:equiv} implies the feasibility of those of Theorem \ref{th:geromel}. The proof is complete.

%
\end{proof}

\begin{remark}
It is worth noting that condition (\ref{eq:expless}) is infinite-dimensional since the decision variables $Z_{ij}(\tau)$'s are matrix functions. In order to render the problem tractable, these functions will be assumed as polynomials and determined using sum-of-squares programming \cite{Chesi:09, Parrilo:00}. This can be performed conveniently using the package SOSTOOLS \cite{sostools} together with the semidefinite programming solver SeDuMi \cite{Sturm:01a}.
\end{remark}
\begin{remark}
The above result provides a sufficient condition for Theorem \ref{th:geromel}. The necessity of the conditions of Theorem \ref{th:equiv}  is an open problem and examples tend to suggest that the conditions might be necessary. This question is left for future research.
\end{remark}

The following corollary concerns the case of constant inter-switching periods:
\begin{corollary}
  Assume there exist a scalar $\eps>0$, matrices $P_i\in\mathbb{S}_{+}^n$, $i=\gint{1}{N}$ and continuously differentiable symmetric matrix functions $Z_{ij}:[0,\bar{T}]\to\mathbb{S}^{3n}$, $i,j=\gint{1}{N}$, $i\ne j$, satisfying the constraints (\ref{eq:loliloolldr}) and such that the LMIs (\ref{eq:expless}) hold for all $\tau\in[0,\bar{T}]$. Then, the switched system (\ref{eq:mainsyst}) with constant switching period $\bar{T}$ is globally asymptotically stable.\mendth
\end{corollary}


\subsection{Connection with looped-functionals}

The  result of Theorem \ref{th:equiv} can be interpreted as a monotonic non-increase condition of the looped-functionals %
\begin{equation}\label{eq:looped}
\begin{array}{lcl}
  W_{ij}(x)&=&x(\tau)^{\transp}P_ix(\tau)+\xi(\tau)^{\transp}Z_{ij}(\tau)\xi(\tau)\\
  &&+\tau x(0)^{\transp}(P_i-P_j+\eps I_n)x(0)
\end{array}
\end{equation}
where $\xi(\tau)=\col(x(\tau),x(0),x(\bar{T}))$, $\tau\in[0,\bar{T}]$ and the matrix functions $Z_{ij}$ satisfy the boundary condition (\ref{eq:loliloolldr}). The looped-functionals (\ref{eq:looped}) are indeed nonincreasing over $[0,\bar{T}]$ if and only if the conditions (\ref{eq:expless}) hold. The term `looped' comes from the presence of the boundary condition (\ref{eq:loliloolldr}) that `loops' both sides of the functional. Such functionals have been successfully applied to the analysis of sampled-data systems \cite{Seuret:12,Seuret:11} and impulsive systems \cite{Briat:11l,Briat:13a}.

\subsection{Examples}

Illustrative examples are given here. The conditions of Theorem \ref{th:equiv} are enforced using sum-of-squares programming \cite{Parrilo:00,Chesi:09} and the semidefinite programming solver SeDuMi \cite{Sturm:01a}. Thus, in the examples below, the matrix functions $Z_{ij}$'s will be searched over the ring of polynomials with fixed degree. Note that polynomials can approximate continuous functions over compact set as precisely as desired by virtue of the Weierstrass approximation theorem. It is thus expected to obtain more and more accurate results as the degree of the $Z_{ij}$'s increases. The rates of convergence of the computed upper bounds may however be very heterogeneous.

\begin{example}
  Let us consider the system (\ref{eq:mainsyst}) with matrices \cite{Geromel:06b}
  \begin{equation}\label{eq:geromelex}
  \begin{array}{lclclcl}
        A_1&=&\begin{bmatrix}
    0 & 1\\
    -10 & -1
    \end{bmatrix},& &A_2&=&\begin{bmatrix}
      0 & 1\\
      -0.1 & -0.5
    \end{bmatrix}.
  \end{array}
  \end{equation}
  Using the initial result on minimal-dwell-time in \cite{Morse:96},  the upper-bound 6.66  on the minimum dwell-time is found. For comparison, the average dwell-time condition of \cite{Hespanha:99} yields the value 16.5554 as the upper-bound on the average dwell-time. Using the minimum dwell-time result based on mixed continuous-time and discrete-time conditions, i.e. Theorem \ref{th:geromel}, the upper bound 2.7508 on the minimum dwell-time is obtained. This emphasizes the efficiency and importance of Theorem \ref{th:geromel} in dwell-time analysis of linear switched systems. Theorem \ref{th:equiv} yields the minimum dwell-time estimates summarized in Table \ref{tab:ex123}. We can see that the proposed method allows to compute quite closely the upper-bound on the minimum dwell-time obtained with Theorem \ref{th:geromel} as the degree of $Z_{ij}$ increases.
  \end{example}
  \begin{table}
  \centering
    \begin{tabular}{|c|c||c|c|c|c|}
    \hline
      & degree of $Z_{ij}$'s & System (\ref{eq:geromelex}) & System (\ref{eq:chesi1}) & System (\ref{eq:chesi3})\\
      \hline
      \hline
      \multirow{7}{*}{Theorem \ref{th:equiv}}&1 & 8.8537 & 0.7438 & 4.0432\\
      &2 & 3.6310 & 0.6222 & 1.9176\\
      &3 & 3.0362 & -- &     1.9168\\
      &4 & 2.9147 & -- &     1.9167\\
      &5 & 2.7739 & -- &     1.9137\\
      &6 & 2.7545 & -- &     1.9135\\
      \hline
      Theorem \ref{th:geromel} & -- & 2.7508 & 0.6222 & 1.9134\\
      \hline
    \end{tabular}
    \caption{Upper bounds on the minimum dwell-time of Systems (\ref{eq:geromelex}), (\ref{eq:chesi1}) and (\ref{eq:chesi3}) determined using Theorem \ref{th:equiv} for different degrees for the polynomial functions $Z_{ij}$.}\label{tab:ex123}
  \end{table}

\begin{example}
  Let us consider the system (\ref{eq:mainsyst}) with matrices \cite{Chesi:10}
  \begin{equation}\label{eq:chesi1}
  \begin{array}{lclclcl}
        A_1&=&\begin{bmatrix}
    0 & 1\\
    -2 & -1
    \end{bmatrix},& &A_2&=&\begin{bmatrix}
      0 & 1\\
      -9 & -1
    \end{bmatrix}.
  \end{array}
  \end{equation}
  Using Theorem \ref{th:geromel}, the upper bound value 0.6222 on the minimum dwell-time is found. Using then Theorem \ref{th:equiv}, we obtain the sequence of upper bounds of Table \ref{tab:ex123}. We can see that the upper-bound determined using Theorem \ref{th:geromel} can be retrieved by using Theorem \ref{th:equiv} with polynomials $Z_{ij}$ of degree 2.
  \end{example}
%

\begin{example}
    Let us consider the system (\ref{eq:mainsyst}) with matrices \cite{Chesi:10}
  \begin{equation}\label{eq:chesi3}
  \begin{array}{lclclcl}
        A_1&=&\begin{bmatrix}
    -1 & -1 & 1\\
    -1 & -1 & 0\\
    -2 & 1 & -1
    \end{bmatrix},& &A_2&=&\begin{bmatrix}
      -1 & 0 & 6\\
      -2 & -1 & -5\\
      0 & 3 & -1
    \end{bmatrix}.
  \end{array}
  \end{equation}
  Using Theorem \ref{th:geromel}, the minimum dwell-time upper bound value 1.9134 is found. Using then Theorem \ref{th:equiv}, we obtain the sequence of upper bounds of Table \ref{tab:ex123}. We can see that by choosing polynomials $Z_{ij}$ of order 6, the result of Theorem \ref{th:geromel} is almost retrieved.
\end{example}

%
%
%

%
%
%

\section{Affine characterization of mode-dependent dwell-times}\label{sec:modedep}

Affine sufficient conditions for Theorem \ref{th:second} are stated below:
\begin{theorem}\label{th:equivb2}
Assume there exist a scalar $\eps>0$, matrices $P_i\in\mathbb{S}^n_{+}$, $i=\gint{1}{N}$ and symmetric matrix functions $Z_{ij}:[0,T^{max}_i]\times[T^{min}_i,T^{max}_i]\to\mathbb{S}^{3n}$, $i,j=\gint{1}{N}$, $i\ne j$, differentiable with respect to the first variable and verifying
\begin{equation}
Y_2^{\transp}Z_{ij}(T_i,T_i)Y_2-Y_1^{\transp}Z_{ij}(0,T_i)Y_1=0
\end{equation}
for all $T_i\in[T^{min}_i,T^{max}_i]$ where
\begin{equation}
\begin{array}{lclclcl}
  Y_1&=&\begin{bmatrix}
    I_n & 0_n\\
    I_n & 0_n\\
    0_n & I_n
  \end{bmatrix},&\quad& Y_2&=&\begin{bmatrix}
    0_n & I_n\\
    I_n & 0_n\\
    0_n & I_n
  \end{bmatrix}
\end{array}
\end{equation}
such that the LMIs
\begin{equation}\label{eq:explessb}
\begin{array}{l}
  \Psi_{ij}(T_i)+\He\left(Z_{ij}(\tau,T_i)\mathcal{D}_n(A_i)\right)\\
 \qquad+\dfrac{\partial Z_{ij}}{\partial \tau}(\tau,T_i)\preceq0,\ i,j=\gint{1}{N},\ i\ne j
\end{array}
\end{equation}
hold for all $\tau\in[0,T_i]$, $T_i\in[T^{min}_i,T^{max}_i]$ where
\begin{equation}
  \Psi_{ij}(T_i):=\begin{bmatrix}
   T_i(A_i^{\transp}P_i+P_iA_i) & 0_n & 0_n\\
   0_n & P_i-P_j+\eps I_n & 0_n\\
   \star & \star & 0_n
 \end{bmatrix}.
\end{equation}
Then, the switched system (\ref{eq:mainsyst}) with mode-dependent dwell-times $T_i\in[T^{min}_i,T^{max}_i]$, $i=1,\ldots,N$ is globally asymptotically stable and the conditions of Theorem \ref{th:second} are satisfied with the same matrices $P_i$'s.\mendth
\end{theorem}
\begin{proof}
  The proof follows the same lines as the one of Theorem \ref{th:equiv}.
\end{proof}

Let us illustrate the above result with an example:
\begin{example}
  Consider the switched system (\ref{eq:mainsyst}) with 2 modes and matrices
  \begin{equation}\label{eq:syst2jdl}
    A_1=\begin{bmatrix}
      -2 & 1\\
     5 & -3
    \end{bmatrix},\quad A_2=\begin{bmatrix}
    0.1 & 0\\
      0.1 & 0.2
    \end{bmatrix}.
  \end{equation}
  The first subsystem is asymptotically stable while the second one is anti-stable. Therefore, (mode-dependent) average dwell-time results such as the ones in \cite{Hespanha:99,Zhao:12} are clearly not applicable since they require that the subsystems be asymptotically stable. Now let $T_1\in[T^{min}_1,\infty)$ and let us determine the range of $T_2\in[T^{min}_2,T^{max}_2]$ such that the overall system is asymptotically stable. Since the first subsystem satisfies a minimum dwell-time condition, Remark \ref{rem:1} applies and the conditions (\ref{eq:rem1_1}) and (\ref{eq:rem1_2}) are considered for mode 1. The mode-dependent dwell-time of mode 2 belongs to a compact interval, hence it must be characterized using the affine conditions of Theorem \ref{th:equivb2}. Setting  $T^{min}_2=0.001$, we obtain the results of Table \ref{tab:range}. Note that in the first two cases, the computed maximal $T^{max}_2$ is equal to the one obtained in the periodic switching case (necessary condition). We can hence conclude on the nonconservatism of the approach for these specific cases. Note also that the obtained results are valid in both the cases of constant and uncertain dwell-times, and time-varying dwell-times.
\end{example}

\begin{table*}
  \centering
  \begin{tabular}{|c|c||c|c|c|c|}
  \hline
   & degree of $Z_{ij}$ & $T^{min}_1=1$ &  $T^{min}_1=2$ & $T^{min}_1=5$ & $T^{min}_1=7$\\
    \hline
   \multirow{2}{*}{Theorem \ref{th:equivb2}}  &1 & 1.2841 & 2.5388 & 5.9931 & 7.8897\\ 
    &2 & 1.2847 & 2.5471 & 6.2149 & 8.5753\\ 
    \hline
    Periodic switching case & -- & 1.2847 & 2.5471 & 6.2158 & 8.5804\\
    \hline
  \end{tabular}
  \caption{Maximal $T_2^{max}$ for system (\ref{eq:syst2jdl}) computed for different values of $T^{min}_1$.}\label{tab:range}
\end{table*}

\section{Minimum dwell-times for uncertain switched systems}\label{sec:unc}

Unlike Theorem \ref{th:geromel}, Theorem \ref{th:equiv} can easily be extended to deal with uncertain systems thanks to the affine dependence of the conditions on the system matrices. Let us assume now that the matrices of the system (\ref{eq:mainsyst}) are uncertain and belong to the convex sets
\begin{equation}\label{eq:uncertainty}
\mathcal{A}_i:=\left\{F_i+U_i\Delta_iV_i:||\Delta_i||_2\le 1\right\},\ i=\gint{1}{N}
\end{equation}
where $F_i$, $U_i$ and $V_i$ are known matrices of appropriate dimensions. The uncertain matrices $\Delta_i$ are allowed to be time-varying. Theorem \ref{th:equiv} then extends naturally to the uncertain case as shown below:
\begin{theorem}\label{th:equiv2}
Assume there exist matrices $P_i\in\mathbb{S}^n_{+}$, $i=\gint{1}{N}$, real symmetric differentiable matrix functions $Z_{ij}:[0,\bar{T}]\to\mathbb{S}^{2n}$, $i,j=\gint{1}{N}$, $i\ne j$, scalar functions $\mu_{ij}:[0,\bar{T}]\to\mathbb{R}_{+}$, $i,j=\gint{1}{N}$, $i\ne j$ and constant scalars $\eps,\mu_i>0$, $i=\gint{1}{N}$, verifying
\begin{equation}\label{eq:loliloolldr2}
Y_2^{\transp}Z_{ij}(\bar{T})Y_2-Y_1^{\transp}Z_{ij}(0)Y_1=0
\end{equation}
where
\begin{equation}
\begin{array}{lclclcl}
  Y_1&=&\begin{bmatrix}
    I_n & 0_n\\
    I_n & 0_n\\
    0_n & I_n
  \end{bmatrix},&\quad& Y_2&=&\begin{bmatrix}
    0_n & I_n\\
    I_n & 0_n\\
    0_n & I_n
  \end{bmatrix}
\end{array}
\end{equation}
and such that the LMIs
\begin{equation}
  \begin{bmatrix}
    F_i^{\transp}P_i+P_iF_i+\mu_i V_i^{\transp}V_i & P_iU_i\\
    \star & -\mu_i I_n
  \end{bmatrix}\prec0
\end{equation}
\begin{equation}\label{eq:expless2}
\begin{bmatrix}
   \Xi_{ij}^1&\vline & \begin{bmatrix}
   Z_{ij}^{11}(\tau)+TP_i\\
   Z_{ij}^{21}(\tau)\\
   Z_{ij}^{31}(\tau)
 \end{bmatrix}U_i\\
 \hline
 \star & \vline & -\mu_{ij}(\tau) I_n
\end{bmatrix}\preceq0
\end{equation}
hold for all $i,j=\gint{1}{N},\ i\ne j$ and all $\tau\in[0,\bar{T}]$ where
\begin{equation}
\begin{array}{lcl}
    \Xi_{ij}^1&=&\Psi_{ij}(\bar{T})+\He\left(Z_{ij}(\tau)\mathcal{D}_n(F_i)\right)\\
    &&+\mathcal{D}_n(\mu_{ij}(\tau)V_i^{\transp}V_i)+\dot{Z}_{ij}(\tau)
\end{array}
\end{equation}
and
\begin{equation}
  \Psi_{ij}(\bar{T}):=\begin{bmatrix}
   \bar{T}\left(F_i^{\transp}P_i+P_iF_i\right) & 0_n & 0_n\\
   0_n & P_i-P_j+\eps I_n & 0_n\\
   0_n & 0_n & 0_n
 \end{bmatrix}
\end{equation}
and $Z_{ij}^{k\ell}$ is the $(k,\ell)$ block of dimension $n$ of matrix $Z_{ij}$. Then, the switched system (\ref{eq:mainsyst})-(\ref{eq:uncertainty}) is globally asymptotically stable for any sequence of switching instants in $\mathbb{I}_{\bar{T}}$ and Theorem \ref{th:geromel} is satisfied for all $A_i\in\mathcal{A}_i$, $i=\gint{1}{N}$.
\end{theorem}
\begin{proof}
The proof is very standard for dealing with this type of uncertainties and is thus only sketched. Substitute first the uncertain system matrices into the LMIs of Theorem \ref{th:equiv}. Then, by applying Petersen's Lemma \cite{Petersen:87a} (or equivalently the Scaled Bounded-Real Lemma with full-block uncertainty, see e.g. \cite{Packard:93,Apkarian:95a}) the uncertain matrices $\Delta_i$ can be eliminated from the LMIs and new conditions involving the scalings $\mu_i$ and $\mu_{ij}(\tau)$ are obtained. A Schur complement on the resulting conditions finally yields those stated in the theorem.
\end{proof}

\begin{example}
We revisit here system (\ref{eq:chesi1}) where the system matrices now belong to the convex sets
\begin{equation}\label{eq:uncds}
   \mathcal{A}_i=\left\{F_i+\kappa\delta_iU_iV_i, |\delta_i|\le 1\right\}
\end{equation}
where the $F_i$'s are equal to the $A_i$'s defined in (\ref{eq:chesi1}), $U_i=\begin{bmatrix}
  1 & 0
\end{bmatrix}^{\transp}$, $V_i=\begin{bmatrix}
  1 & 0
\end{bmatrix}$ and $\delta_i\in[-1,1]$, $i=1,2$. The additional parameter $\kappa>0$ is the maximal amplitude of the perturbation. Using a gridding approach, the LMI conditions of Theorem \ref{th:geromel} indicate that the maximal $\kappa$ for which the LMIs are still feasible is $\kappa_{max}=1.3229$. Computed upper-bounds on the minimum dwell-time according to different values for $\kappa>0$ and different degrees for $Z_{ij}$ are given in Table \ref{tab:ex4}. It is interesting to note that the accuracy of the approach reduces when the perturbation magnitude $\kappa$ increases. This can be understood by the fact that the looped-functional does not depend on the uncertain parameter and it is more and more difficult to find a \emph{common looped-functional} as the maximal amplitude of the uncertainty increases. This problem may be solved by making the functional depending on the parameters as this is usually done in robust/LPV analysis. Note also that gridding the conditions of Theorem \ref{th:geromel} is very imprecise (only checks a finite number of points) and has high computational complexity, while the proposed approach allows to consider all the possible matrices in the uncertainty set. An advantage of Theorem \ref{th:equiv2} over Theorem \ref{th:geromel}, is that the results are also valid in the case of time-varying parameters/matrices, therefore the results of Table \ref{tab:ex4} are valid for arbitrarily time-varying parameters $\delta_i(t)\in[-1,1]$.
  \begin{table*}
  \centering
    \begin{tabular}{|c|c||c|c|c|c|c|c|c|}
    \hline
      & degree of $Z_{ij}$ & $\kappa=0.1$ & $\kappa=0.3$ & $\kappa=0.5$ & $\kappa=0.7$ & $\kappa=0.9$ & $\kappa=1.1$ & $\kappa=1.3$\\
      \hline
      \hline
      \multirow{4}{*}{Theorem \ref{th:equiv2}}&1 & 0.8359 & 1.0379 & 1.2691 & 1.5756 & 2.0605 & 2.9498 & 5.6306\\
      &2 & 0.6807 & 0.7941 & 0.9288 & 1.1412 & 1.4614 & 1.9617 & 2.8692\\
      &3 & 0.6788 & 0.7425 & 0.8008 & 0.8844 & 1.4521 & 1.3146 & 2.1973\\
      &4 & 0.6785 & 0.7418 & 0.7988 & 0.8803 & 1.0113 & 1.2038 & 1.8835\\
      &5 & 0.6785 & 0.7413 & 0.7976 & 0.8786 & 1.0004 & 1.1834 & 1.7174\\
      \hline
      Theorem \ref{th:geromel}& -- & 0.6759 &  0.7298 & 0.7689 & 0.8128 &  0.8673 & 0.9512 & 1.1475\\
      \hline
    \end{tabular}
      \caption{Upper bounds on the minimum dwell-time of system (\ref{eq:chesi1})-(\ref{eq:uncds}) determined using Theorem \ref{th:equiv2} for different degrees for $Z$}\label{tab:ex4}
  \end{table*}
\end{example}

\section{Conclusion}

New conditions for minimal and mode-dependent dwell-times characterization for linear switched systems have been presented. The affine structure of the conditions has allowed to extend the results to the uncertain case. Several examples illustrate the approach. Future works will be devoted to the necessity analysis of the obtained conditions.

\bibliographystyle{IEEEtran}

\begin{thebibliography}{10}
\providecommand{\url}[1]{#1}
\csname url@samestyle\endcsname
\providecommand{\newblock}{\relax}
\providecommand{\bibinfo}[2]{#2}
\providecommand{\BIBentrySTDinterwordspacing}{\spaceskip=0pt\relax}
\providecommand{\BIBentryALTinterwordstretchfactor}{4}
\providecommand{\BIBentryALTinterwordspacing}{\spaceskip=\fontdimen2\font plus
\BIBentryALTinterwordstretchfactor\fontdimen3\font minus
  \fontdimen4\font\relax}
\providecommand{\BIBforeignlanguage}[2]{{%
\expandafter\ifx\csname l@#1\endcsname\relax
\typeout{** WARNING: IEEEtran.bst: No hyphenation pattern has been}%
\typeout{** loaded for the language `#1'. Using the pattern for}%
\typeout{** the default language instead.}%
\else
\language=\csname l@#1\endcsname
\fi
#2}}
\providecommand{\BIBdecl}{\relax}
\BIBdecl

\bibitem{Hespanha:99}
J.~P. Hespanha and A.~S. Morse, ``Stability of switched systems with average
  dwell-time,'' in \emph{38th Conference on Decision and Control}, Phoenix,
  Arizona, USA, 1999.

\bibitem{Hespanha:01c}
J.~P. Hespanha, ``Extending {L}asalle's invariance principle to switched linear
  systems,'' in \emph{Decision and Control, Orlando, Florida, USA}, vol.~3,
  2001, pp. 2496--2501.

\bibitem{Daafouz:02}
J.~Daafouz, P.~Riedinger, and C.~Iung, ``Stability analysis and control
  synthesis for switched systems: A switched {L}yapunov function approach,''
  \emph{IEEE Transactions on Automatic Control}, vol. 47(11), pp. 1883--1887,
  2002.

\bibitem{Liberzon:03}
D.~Liberzon, \emph{Switching in Systems and Control}.\hskip 1em plus 0.5em
  minus 0.4em\relax Birkh{\"{a}}user, 2003.

\bibitem{Geromel:06b}
J.~Geromel and P.~Colaneri, ``Stability and stabilization of continuous-time
  switched linear systems,'' \emph{{SIAM} Journal on Control and Optimization},
  vol. 45(5), pp. 1915--1930, 2006.

\bibitem{Zhao:08}
J.~Zhao and D.~J. Hill, ``Dissipativity theory for switched systems,''
  \emph{IEEE Transactions on Automatic Control}, vol. 53(4), pp. 941--953,
  2008.

\bibitem{Cai:08}
C.~Cai, A.~R. Teel, and R.~Goebel, ``Smooth {L}yapunov functions for hybrid
  systems {P}art ii: (pre)asymptotically stable compact sets,'' \emph{IEEE
  Transactions on Automatic Control}, vol. 53(3), pp. 734--748, 2008.

\bibitem{Lin:09}
H.~Lin and P.~J. Antsaklis, ``Stability and stabilizability of switched linear
  systems: A survey of recent results,'' \emph{{IEEE} Transactions on Automatic
  Control}, vol. 54(2), pp. 308--322, 2009.

\bibitem{Goebel:09}
R.~Goebel, R.~G. Sanfelice, and A.~R. Teel, ``Hybrid dynamical systems,''
  \emph{{IEEE} Control Systems Magazine}, vol. 29(2), pp. 28--93, 2009.

\bibitem{Hespanha:01b}
J.~Hespanha, S.~Bohacek, K.~Obraczka, and J.~Lee, ``Hybrid modeling of {TCP}
  congestion control,'' in \emph{Hybrid Systems: Computation and Control}, ser.
  Lecture Notes in Computer Science, M.~Di~Benedetto and
  A.~Sangiovanni-Vincentelli, Eds.\hskip 1em plus 0.5em minus 0.4em\relax
  Springer Berlin / Heidelberg, 2001, vol. 2034, pp. 291--304.

\bibitem{Shorten:06}
R.~Shorten, F.~Wirth, and D.~Leith, ``A positive systems model of {TCP}-like
  congestion control: asymptotic results,'' \emph{{IEEE} Transactions on
  Networking}, vol. 14(3), pp. 616--629, 2006.

\bibitem{Briat:10}
C.~Briat, H.~Hjalmarsson, K.~H. Johansson, G.~Karlsson, U.~T. {J\"{o}nsson},
  and H.~Sandberg, ``Nonlinear state-dependent delay modeling and stability
  analysis of internet congestion control,'' in \emph{49th IEEE Conference on
  Decision and Control}, Atlanta, USA, 2010, pp. 1484--1491.

\bibitem{Briat:11k}
C.~Briat, H.~Hjalmarsson, K.~H. Johansson, G.~Karlsson, U.~T. {J\"{o}nsson},
  H.~Sandberg, and E.~A. Yavuz, ``An axiomatic fluid-flow model for congestion
  control analysis,'' in \emph{50th {IEEE} Conference on Decision and Control},
  Orlando, Florida, USA, 2011, pp. 3122--3129.

\bibitem{Langjord:08}
H.~Langjord, T.~A. Johansen, and J.~P. Hespanha, ``Switched control of an
  electropneumatic clutch actuator using on/off valves,'' in \emph{American
  Control Conference, Seattle, Washington, USA}, june 2008, pp. 1513--1518.

\bibitem{Donkers:09}
M.~Donkers, L.~Hetel, W.~Heemels, N.~van~de Wouw, and M.~Steinbuch, ``Stability
  analysis of networked control systems using a switched linear systems
  approach,'' in \emph{Hybrid Systems: Computation and Control}, ser. Lecture
  Notes in Computer Science, R.~Majumdar and P.~Tabuada, Eds.\hskip 1em plus
  0.5em minus 0.4em\relax Springer Berlin / Heidelberg, 2009, vol. 5469, pp.
  150--164.

\bibitem{Almer:07}
S.~Alm{\'{e}}r, H.~Jujioka, U.~T. J{\"{o}}nsson, C.-Y. Kao, D.~Patino,
  P.~Riedinger, T.~Geyer, A.~Beccuti, G.~Papafotiou, M.~Morari, A.~Wernrud, and
  A.~Rantzer, ``Hybrid control techniques for switched-mode {DC-DC} converter -
  part {I}: The step-down topology,'' in \emph{American Control Conference, New
  York, USA}, 2007, pp. 5450--5457.

\bibitem{Beccuti:07}
A.~G. Beccuti, G.~Papafotiou, M.~Morari, S.~Alm{\'{e}}r, H.~Fujioka,
  U.~J{\"{o}}nsson, C.-Y. Kao, A.~Wernrud, A.~Rantzer, M.~B{\^{a}}ja,
  H.~Cormerais, and J.~Buisson, ``Hybrid control techniques for switched-mode
  {DC-DC} converters - part {II}: The step-up topology,'' in \emph{American
  Control Conference, New York, USA}, 2007, pp. 5464--5471.

\bibitem{Decarlo:00}
R.~A. Decarlo, M.~S. Branicky, S.~Pettersson, and B.~Lennartson, ``Perspectives
  and results on the stability and stabilizability of hybrid systems,'' in
  \emph{Proceedings of the IEEE}, vol. 88(7), 2000, pp. 1069--1082.

\bibitem{Liberzon:99}
D.~Liberzon, J.~P. Hespanha, and A.~S. Morse, ``Stability of switched systems:
  a {L}ie-algebraic condition,'' \emph{System and Control Letters}, vol.~37,
  pp. 117--122, 1999.

\bibitem{Molchanov:89}
A.~Molchanov and Y.~Pyatnitskiy, ``Criteria of asymptotic stability of
  differential and difference inclusions encountered in control theory,''
  \emph{Systems \& Control Letters}, vol.~13, pp. 59--64, 1989.

\bibitem{Blanchini:95}
F.~Blanchini, ``Nonquadratic {L}yapunov functions for robust control,''
  \emph{Automatica}, vol. 31(3), pp. 451--461, 1995.

\bibitem{Morse:96}
A.~S. Morse, ``Supervisory control of families of linear set-point controllers
  - part 1: {E}xact matching,'' \emph{IEEE Transactions on Automatic Control},
  vol. 41(10), pp. 1413--1431, 1996.

\bibitem{Chesi:10}
G.~Chesi, P.~Colaneri, J.~C. Geromel, R.~Middleton, and R.~Shorten, ``Computing
  upper-bounds of the minimum dwell time of linear switched systems via
  homogeneous polynomial {L}yapunov functions,'' in \emph{American Control
  Conference}, Baltimore, Maryland, USA, 2010, pp. 2487--2492.

\bibitem{Zhao:12}
X.~Zhao, L.~Zhang, P.~Shi, and M.~Liu, ``Stability and stabilization of
  switched linear systems with mode-dependent average dwell-time,'' \emph{To
  appear in {IEEE} Transactions on Automatic Control}, 2012.

\bibitem{Mignone:00}
D.~Mignone, G.~Ferrari-Trecate, and M.~Morari, ``Stability and stabilization of
  piecewise affine and hybrid systems,'' in \emph{39th {IEEE} Conference on
  Decision and Control}, 2000, pp. 504--509.

\bibitem{Seuret:12}
A.~Seuret, ``A novel stability analysis of linear systems under asynchronous
  samplings,'' \emph{Automatica}, vol. 48(1), pp. 177--182, 2012.

\bibitem{Seuret:11}
A.~Seuret and M.~Peet, ``{SOS} for sampled-data systems,'' in \emph{18th {IFAC}
  World Congress}, Milano, Italy, 2011, pp. 1441--1446.

\bibitem{Briat:11l}
C.~Briat and A.~Seuret, ``A looped-functional approach for robust stability
  analysis of linear impulsive systems,'' \emph{To appear in Systems \& Control
  Letters}, 2012.

\bibitem{Briat:13a}
------, ``Convex dwell-time characterizations for uncertain linear impulsive
  systems,'' \emph{to appear in {IEEE} Transactions on Automatic Control
  (January 2013)}, 2013.

\bibitem{Parrilo:00}
P.~Parrilo, ``Structured semidefinite programs and semialgebraic geometry
  methods in robustness and optimization,'' Ph.D. dissertation, California
  Institute of Technology, Pasadena, California, 2000.

\bibitem{Chesi:09}
G.~Chesi, A.~Garulli, A.~Tesi, and A.~Vicino, \emph{Homogeneous polynomial
  forms for robustness analysis of uncertain systems}.\hskip 1em plus 0.5em
  minus 0.4em\relax Springer-Verlag, 2009.

\bibitem{sostools}
S.~Prajna, A.~Papachristodoulou, P.~Seiler, and P.~A. Parrilo,
  \emph{{SOSTOOLS}: Sum of squares optimization toolbox for {MATLAB}},
  Available from \texttt{http://www.cds.caltech.edu/sostools} and
  \texttt{http://www.mit.edu/\~{}parrilo/sostools}, 2004.

\bibitem{Sturm:01a}
J.~F. Sturm, ``Using sedumi $1.02$, a matlab toolbox for optimization over
  symmetric cones,'' \emph{Optimization Methods and Software}, vol.~11, no.~12,
  pp. 625--653, 2001.

\bibitem{Petersen:87a}
I.~Petersen, ``A stabilization algorithm for a class of uncertain linear
  systems,'' \emph{Systems \& Control Letters}, vol.~8, pp. 351--357, 1987.

\bibitem{Packard:93}
A.~Packard and J.~C. Doyle, ``The complex structured singular value,''
  \emph{Automatica}, vol.~29, pp. 71--109, 1993.

\bibitem{Apkarian:95a}
P.~Apkarian and P.~Gahinet, ``A convex characterization of gain-scheduled
  $\mathcal{H}_\infty$ controllers,'' \emph{IEEE Transactions on Automatic
  Control}, vol.~5, pp. 853--864, 1995.

\end{thebibliography}


\end{document}